\documentclass{article}
\usepackage{graphicx} 
\usepackage{amssymb}
\usepackage{amsmath}
\usepackage{amscd}
\usepackage{mathtools}
\usepackage{systeme}
\usepackage{amsthm}
\usepackage{color}
\usepackage[
top    = 3.00cm,
bottom = 3.00cm,
left   = 3.00cm,
right  = 2.80cm]{geometry}
\usepackage[style=alphabetic,natbib=true,bibencoding=ascii,backend=bibtex]{biblatex}
\usepackage{hyperref}
\usepackage{etoolbox}
\usepackage{tikz-cd}
\usepackage{faktor}
\usepackage{makeidx}
\usepackage{tikz-cd}
\newtheorem{lemma}{Lemma}

\newtheorem{remark}{Remark}
\newtheorem{prop}{Proposition}
\newtheorem{theorem}{Theorem}

\title{Lagrange's theorem for a family of finite flat group schemes over local Artin rings}
\author{Emiliano Torti}
\date{Laboratoire GAATI, Universit\'e de la Polyn\'esie Fran\c{c}aise\\[2ex]
torti.emiliano@gmail.com\\
\today
}

\begin{document}

\maketitle
\begin{abstract}\noindent
Let $R$ be a local Artin ring with residue field $k$ of positive characteristic. We prove that every finite flat group scheme over $R$ whose special fiber belongs to a certain explicit family of non-commutative $k$-group schemes is killed by its order. This is achieved via a classification result which rely on the explicit study of the infinitesimal deformation theory for such non-commutative $k$-group schemes. The main result answers positively in a new case a question of Grothendieck in SGA 3 on whether all finite flat group schemes are killed by their order. 
\end{abstract}

\section{Introduction}
Let $S$ be a locally Noetherian base scheme and let $G$ be an $S$-scheme. Denote by $\mathcal{O}_S$ and $\mathcal{O}_G$ their respective structure sheaves. The $S$-scheme $G$ is finite and flat if and only if $\mathcal{O}_G$ is a locally free $\mathcal{O}_S$-module of finite rank. A standard EGA reduction argument (see section 8.9,  8.10 and especially section 9.2 in EGA IV Tome 3 in \cite{EGA}), allows us to restrict locally and assume that $S=\text{Spec} (R)$ for some local Noetherian ring $R$ and $G=\text{Spec} (A)$ where $A$ is a finite and free $R$-module of finite rank, say the positive integer $n_R$. By functoriality, the positive integer $n_R$ is the restriction to $\text{Spec}(R)$ of a locally constant function say $n$ defined on $S$ with positive integer values. Such function is called the order of the finite flat $S$-scheme $G$. \\
Assume now that $G=\text{Spec}(A)$ over $\text{Spec}(R)$ has the extra structure of an $R$-group scheme. Note that in notation and terminology in this article we will say $R$-group scheme instead of $\text{Spec}(R)$-group scheme. The category of affine $R$-group schemes is anti-equivalent to the category of commutative Hopf $R$-algebras. This implies that $A$ has a natural structure of Hopf algebra over $R$. Let now $n$ be the order of $G=\text{Spec}(A)$ over $R$, i.e. the rank of $A$ as a finite and free $R$-module. We say that $G$ is killed by its order if the multiplication-by-$n$ morphism $[n] : G \rightarrow G$ is the zero morphism of $R$-group schemes. This is equivalent to say that the induced morphism of Hopf $R$-algebras $ [n] : A \rightarrow A^{\otimes n}\rightarrow A$ (obtained via a composition of the diagonal map and the multiplication map on $G$) satisfies that its kernel $\text{Ker} ([n])$ contains the augmentation ideal $I=\text{Ker}(\varepsilon: A \rightarrow R)$ where $\varepsilon$ denotes the unit section of $G$. \\
Around 1960s, Grothendieck asked if every finite flat group scheme over any base scheme is killed by its order (see SGA 3, Exp. VIII Rem. 7.3.1 in \cite{SGA3x} , see also pag. 145 in \cite{Ta97} and the survey article \cite{Sch19}). This is the generalization in the context of finite flat group schemes of what is well-known in abstract group theory, i.e. if $G$ is a finite group of order $n$ then for all $g\in G$ we have $g^n =e_G$. Note that strictly speaking this fact is an immediate consequence of Lagrange's theorem, which has indeed been generalized in the context of finite flat group schemes (see for example \cite{Sha86}). We clarify this in the hope that the title of the article will be interpreted properly.\\
Grothendieck's question has been positively answered in many important cases. When the group scheme $G$ is commutative, this has been proven by Deligne around 1970. The proof has not been published by Deligne himself but it has been reproduced many times in the literature (see for example section 1 in \cite{OT70}, section 3.8 in \cite{Ta97} and section 3.3 in J. Stix's lecture notes \cite{Sti09}). When the base scheme is the spectrum of a field, or more generally is a reduced scheme, this has been proven by Grothendieck in SGA 3 Exp. VII A  section 8 (see \cite{SGA3} or see also Cor. 2.2 in \cite{Sch01}).\\
When studying Grothendieck's question for finite flat group schemes over a local Noetherian ring $R$, it is very useful to notice that it is possible to make a further reduction. First, by Krull's intersection theorem, we deduce that the question holds if and only if it holds for all local Artin quotients of $R$, i.e. we can assume without loss of generality that $R$ is a local Artin ring. This simplification is known to be very useful as it opens up the possibility of proceeding by induction on the lenght of $R$, as it will be done in this article. Moreover, it is clear that the problem of understanding when a finite flat group scheme $G$ over $R$ is killed by its order can be studied up to faithfully flat extension of $R$. Thanks to EGA 3 Tome 1 Prop. 10.3.1 in \cite{EGA1}, we have that there exists a faithfully flat local extension of $R$ such that its residue field $k$ is algebraically closed. This allows one to do a final reduction and assume that the residue field $k$ of characteristic $p$ is algebraically closed and that $R$ is a strictly Henselian ring. For the moment, without loss of generality we assume only that $k$ is a perfect field of characteristic $p$. \\
Coming back to the state of art of Grothendieck's question, the best known result in the general case is due to Schoof (see \cite{Sch01}):
\begin{theorem}[Schoof, 2001]
Let $p$ be a prime and let $R$ be a local Artin ring with maximal ideal $\mathfrak{m}_R$ and with residue field $k$ of positive characteristic $p$. Assume that $\mathfrak{m}_R^p = p\mathfrak{m}_R =0$, than any finite flat group scheme $G$ over $R$ is killed by its order. 
\end{theorem}
\noindent
After reducing the problem to assuming that $R$ is strictly Henselian and $k$ is an algebraically closed field of characteristic $p$, the above result is proven by discussing separately two cases depending on whether the base change $G\otimes_R k$ of $G$ from $R$ to its residue field $k$ (which is a $k$-group scheme of order say $p^{m+1}$, just for psychological reasons and notation consistency) is killed or not by $p^{m}$.\\
Schoof proceeds by showing first that under the strong hypothesis on the lenght of $R$, the groups $G$ whose base change $G\otimes_R k$ is already killed by $p^{m}$ are killed by their order. The proof continues by showing that $G\otimes_R k$ is not killed by $p^{m}$ if and only if $G\otimes_R k$ is isomorphic to $\mu_{p^{m}} / k $ or to the matrix $k$-group scheme 
$$G_1 = \begin{pmatrix}1 & \alpha_p \\ 0 &\mu_{p^{m}}\end{pmatrix} =\Bigg\{ \begin{pmatrix} 1 & x \\ 0 & y\end{pmatrix} \text{ : }x^p=0 , \;y^{p^{m}}=0 \Bigg\},$$
endowed with the usual matrix multiplication. It is interesting to remark that these groups are historically very important in the study of finite group schemes. Indeed, for example when $m=1$, we have the non-commutative $k$-group scheme $G_1$ of order $p^2$ whose existence reflects the existence of a non-trivial action of $\mu_p$ on $\alpha_p$ (which are both commutative groups of order $p$). This phenomenon does not have any analogue in the classical theory of finite abstract groups where every group of order $p^2$ is commutative. \\ 
Note that the group $G_1$ is denoted $G_0$ in \cite{Sch01} but it will be clear soon enough why we adopt this new notation and why it is not an explicit tentative to confuse the reader. 
Finally, Schoof concludes by showing that all deformations of $G_1$ and $\mu_{p^m}$ for some positive integer $m$ are killed by their order. \\
In other words, because finite flat deformations of $\mu_{p^m}$ over local Artin rings $R$ are well-known to be trivial (as we see will see later on), Schoof's strategy is essentially reduced to the following result for which no restrictions are needed on the length of the local Artin ring $R$: 
\begin{theorem}[Schoof, 2001]
Let $p$ be a prime and let $R$ be a local Artin ring of residue field $k$ of positive characteristic $p$. Any finite flat $R$-group scheme $G$ such that $G \otimes_R k \cong G_1$ is killed by its order. 
\end{theorem}
\noindent
We define now a certain family of finite flat group schemes $G_\lambda$ over $k$ where $\lambda \in \{1, \cdots p^{m}-1 \}$ and which interpolates $G_1 = \begin{pmatrix}1 & \alpha_p \\ 0 &\mu_{p^{m}}\end{pmatrix}$ exactly when $\lambda=1$.\\
Let $G$ be any finite flat multiplicative group scheme over $k$ acting on $\alpha_p$, i.e. we have a non-trivial morphism $\varphi : G \rightarrow \text{Aut}_k (\alpha_p )\cong \mathbb{G}_m$. Then there exists a unique extension of $G$ by $\alpha_p$, i.e. $\text{Ext}_\varphi^1 (G, \alpha_p )\cong 0$ (see Cor. 6.4, Chapt. 3 in \cite{DG70}). This is a consequence of the fact that group extensions can be embedded in the group of extensions of fpqc sheaves (see Rem. 2.5 in \cite{DG70}) and such group is trivial under the assumption that $k$ is perfect. More generally, this holds when the base scheme is the spectrum of a small $k$-ring R (in the sense of Demazure and Gabriel see \cite{DG70}) such that $R/R^p =0$. For more details on the relation between affine group schemes and fpqc sheaves we refer the reader to the nice summary in \cite{Sti09} and to SGA 1 Exp. VI for a more in depth treatment (see for example section 6 in \cite{SGA1}).\\
Taking $G=\mu_{p^m}$ we have that $\varphi$ is the restriction to $\mu_{p^m}$ of the standard action of $\mathbb{G}_m$ on $\mathbb{G}_a$ which corresponds uniquely to an element $\lambda \in \{1, \cdots, p^m -1\}$. In more precise terms, for every $\lambda \in \{1, \cdots, p^m -1\}$ there is a unique finite flat $k$-group scheme $G_\lambda$ corresponding to the unique split exact sequence $0 \rightarrow \alpha_p \rightarrow G_\lambda \rightarrow \mu_{p^m} \rightarrow 0$. In other words, the $k$-group scheme $G_\lambda$ is the semi-direct product $\alpha_p \rtimes_\lambda \mu_{p^m}$ where the subscript $\lambda$ indicates exactly the action of $\mu_{p^m}$ which we are using to build the semi-direct product. To be completely explicit, we have that the $k$-group scheme $G_\lambda$ is described as:
$$G_\lambda \cong \text{Spec} \Big( \faktor{k[x, y]}{(x^p , y^{p^m})} \Big)$$
with group law: 
$$ k[x, y ]/(x^p , y^{p^m}) \rightarrow k[x, y ]/(x^p , y^{p^m}) \otimes_k k[x, y ]/(x^p , y^{p^m})$$
$$x \mapsto (1+y)^\lambda \otimes x + x \otimes 1 $$
$$y \mapsto y \otimes 1 + 1 \otimes y + y\otimes y$$
More concisely, we adopt the notation $(x, y) \circ (x' , y') = ((1+y)^\lambda x' + x , y + y' + yy')$ for the group law in $G$ between generic elements $h=(x, y)$ and $h'=(x' , y') $. It is straightforward to verify that when $\lambda=1$ one gets exactly the group $G_1$ described above which appears in \cite{Sch01}. \\
We show that $G_\lambda$ is a family of finite, flat, non-commutative group schemes of order $p^{m+1} $ which depend only on the $p$-adic valuation of $\lambda$, and that $G_\lambda$ is not killed by $p^m$ if and only if $\lambda$ is a unit. Hence, taking a generic $\lambda$ puts us outside of the range of Schoof's theorem. The fact that $\lambda$ is not in general a unit presents many subtle challenges which do not allow us to simply extend the proof of Schoof and force us to find a new and different strategy, especially when one wants to determine explicitly the $R$-scheme theoretical structure of the generic deformation of $G_\lambda$. However, it is possible to find a way to extend Schoof's result without restriction on $R$ and have the following:

\begin{theorem}\label{Lagrange}
Let $R$ be a local Artin ring of positive residue characteristic $p$ and let $\lambda\in \{1, \cdots , p^{m-1}\}$. Let $G$ be a finite flat deformation over $R$ of the $k$-group scheme $G_\lambda$, then $G$ is killed by its order. 
\end{theorem}
\noindent
This result constitutes a positive answer in a new case to Grothendieck's question on whether all finite flat group schemes over any base scheme are killed by their order. The above theorem is a direct consequence of the complete classification of all the deformations of the non-commutative $k$-group scheme $G_\lambda$. On a side note, we mention that one could also consider the case of $\lambda=0$, which corresponds to the case of the $k$-group scheme $G_0 =\alpha_p \times \mu_{p^m}$. However, while it is possible to use the techniques in this article to show that the deformations of $G_0$ in characteristic $p$ are also killed by their order, this still leaves the problem open. Indeed, the $k$-group scheme $G_0$ admits deformations in characteristic zero (because $\alpha_p$ does) which a priori have not a clear $R$-schematic description. We hope to come back to this problem later on.\\
Going back to our original question, the study of the deformation problem for $G_\lambda$ with $\lambda \in \{1, \cdots, p^m\}$ will allow us to prove the following complete and explicit classification:
\newpage
\begin{theorem}\label{maindef}
Let $p$ be a prime and let $R$ be a local Artin ring of perfect residue field $k$ of positive characteristic $p$. Let $G$ be a deformation over $R$ of the $k$-group scheme $G_\lambda \cong \alpha_p \rtimes_\lambda \mu_{p^m}$  with parameter $\lambda \in \{1, \cdots , p^{m}-1\}$.\\
Then the characteristic of $R$ is $p$ and we have the following classification:\\
\break
\indent(i) if $v_p (\lambda)\not=m-1$, all deformations of $G_\lambda$ over $R$ are trivial, i.e. 
$$G \cong (\alpha_p \rtimes_\lambda \mu_{p^m}) \times_{\text{Spec} (k)} \text{Spec}(R) ; $$
\indent(ii) if $v_p (\lambda)=m-1$, all deformations of $G_{p^{m-1}}$ over $R$ form a 1-dimensional family (over $k$) of\\
\indent\indent\indent\indent\indent\indent\indent non-commutative finite flat $R$-group schemes of order $p^{m+1}$.\\
\break
\indent In particular, in (ii), there is a unique $a\in k$ such that we have an isomorphism:
$$G \cong \tilde{H}_{a} \text{ as }R\text{-group schemes, }$$
\indent where $\tilde{H}_{a}\cong \text{Spec}(R[x, y] / (x^p , y^{p^m } ))$ is endowed with the group law:
$$
\begin{aligned}
x &\mapsto (1+y)^{p^{m-1}} \otimes x + x \otimes 1 +a\pi W_p (x\otimes 1 , 1\otimes x),\\
y &\mapsto y\otimes 1 + 1 \otimes y + y \otimes y.
\end{aligned}
$$
for the polynomial $W_p (x, x')=\frac{(x+x')^p -x^p -x'^p}{p}$. 
\end{theorem}
\begin{remark}\normalfont
Note that this results extends to all $\lambda \in\{1, p^{m}-1\} $ the proposition 3.3 in \cite{Sch01} where the specific case $v_p (\lambda)=0$ is treated. Moreover, even when $v_p (\lambda)=0$, Theorem 4 ensures us that all deformations of $G_\lambda$ are trivial over $R$ rather than potentially trivial after base change to a faithful flat extension as shown in proposition 3.3 in \cite{Sch01}.  
\end{remark}
\begin{remark}\normalfont
Note that the groups $\tilde{H}_a$ for $a\in k$ are isomorphic to the semidirect product $H_a \rtimes \mu_{p^m}$ where $H_a$ is the deformation of $\alpha_p$ as described by Prop. \ref{Tate} where the action of $\mu_{p^m}$ on $H_a$ is the one corresponding to $\lambda=p^{m-1}$.
\end{remark}
\noindent
In the first section after the introduction, we prove some useful properties concerning the groups $G_\lambda$ and in the second section we state the main theorem and carry on the study of the deformation problem for the $k$-group schemes $G_\lambda$. We conclude the second section by stating and proving Theorem \ref{Lagrange}.\\
A notation remark, the base change of a $R$-group scheme via a $R$-algebra morphism $R \rightarrow S$ is usually denoted $G \times_{\text{Spec}(R)} \text{Spec}(S)$ but sometimes more concisely we adopt the notation $G \otimes_R S$. \\
\break
\textbf{Acknowledgments:} I would like to thank R. Schoof for sharing his insights on the topic. I would also like to thank A. Vanhaecke and A. Conti for many interesting conversations. I would also like to thank G. Bisson and A. Rahm in the GAATI group at University of French Polynesia who supported my research as part of the MELODIA project. The MELODIA project is funded by ANR under grant number ANR-20-CE40-0013. 

\section{The family of finite flat group schemes $G_\lambda$ over $k$}

In this short section we summarize the main properties of the family $G_\lambda$ of finite flat group schemes over a perfect field $k$ of characteristic $p$. First, note that the $k$-group scheme $G_\lambda$ has order $p^{m+1}$ and it is killed by its order. This can be checked either by a direct computation or by using the well-known fact that any finite flat group scheme over a field is killed by its order (see for example SGA 3, Exp. VIIA Prop. 8.5 in \cite{SGA3}). The family of $k$-group schemes $G_\lambda$ also satisfy the following useful properties:

\begin{lemma}
The $k$-group scheme $G_\lambda$ (of order $p^{m+1}$) is killed by $p^m$ if and only if $v:=v_p (\lambda)\geq 1$. 

\end{lemma}

\begin{proof}
We prove first the "if" part. We recall that the group operation on $G_\lambda$ is as follows:
$$(x, y)\circ_{G_\lambda} (x', y') = (x' (1+y)^\lambda +x, y+y'+yy').$$
By induction, it is straighforward to check that for all $h\in\mathbb{ Z}_{\geq 0}$ we have that:
$$[p^h ] (x, y)= \Big(x \sum_{i=0}^{p^h -1} (1+y)^{\lambda i} , (1+y)^{p^h} -1 \Big) $$
Proving that $p^m$ kills $G_\lambda$ means to prove that $[p^m] (x, y )= (0,0)$ which boils down to prove that: 
$$\sum_{i=0}^{p^m -1} (1+y)^{\lambda i}=0$$
in the ring $k[y]/(y^{p^m})$, because it is clear that $(1+y)^{p^m}-1=0$. Note first that the positive integer $m-v$ is the minimal positive integer $h$ such that that $(1+y)^{\lambda h} =1$.  For $r\leq s$ positive integers denote
$$a(r , s )= \sum_{i=r}^{s} (1+y)^{\lambda i} .$$
The claim is to prove that $a(0, p^{m}-1)=0$. Recall that $v:=v_p (\lambda)$. We have:
$$
a(0, p^m -1 )=a(0, p^{m-v}-1 )+a(p^{m-v} , 2p^{m-v}-1)+\dots+a((p^v -1)p^{m-v} , p^v p^{m-v} -1)=$$
$$=a(0, p^{m-v} -1)+(1+y)^{\lambda p^{m-v}} a(0, p^{m-v} -1)+ \dots + (1+y)^{\lambda (p^{v} -1)p^{m-v}} a(0, p^{m-v} -1)=$$
$$=p^v a(0, p^{m-v} -1)=0
$$
where the last equality holds because $v\geq 1$ and $k$ is of characteristic $p$.\\
Now we prove the "only if" part. Assume by contradiction that $[p^m] (x, y) = (0, 0)$, i.e. that $p^m$ kills $G_\lambda$ (where $v_p (\lambda) =0$, i.e. $\lambda$ is a unit). Since $\psi_\lambda : \mu_{p^m} \rightarrow \mu_{p^m}$ such that $\psi_\lambda (y) = (1+y)^\lambda -1$ is an isomorphism, we have that $\sum_{i=0}^{p^m -1 } (1+y)^{\lambda i} =0$ if and only if $\sum_{i=0}^{p^m -1 } (1+y)^{ i} =0$. Denote by $f(y):= \sum_{i=0}^{p^m -1 } (1+y)^{ i} $. We have that $f(0)=0$ and that the degree $\text{deg} (f) = p^{m-1}$. However, as $k$-vector space we have that $\text{dim}_k (k[y] / (y^{p^m}))= p^m$ so there cannot be any non-trivial zero linear combination of the standard basis and this gives the desired contradiction.
\end{proof}

\begin{lemma}\label{units}
We have that $G_\lambda \cong G_\mu$ if and only if $v_p (\lambda)=v_p (\mu)$.
\end{lemma}

\begin{proof}
Let $u$ be a positive integer (strictly smaller than $p^m$) such that $v_{p} (u)=0$. We have an isomorphism of finite flat $R$-group schemes $\psi_u : \mu_{p^m} \rightarrow \mu_{p^m}$ such that $\psi_u (y)=(1+y)^u -1$, where $y$ is a chosen variable parametrizing $\mu_{p^m}$ (which is endowed with the group law $y \rightarrow y\otimes 1 + 1 \otimes y + y \otimes y$).
Let now $\lambda$ and $\mu:=u \lambda$, so we have $v_p (\lambda)=v_p (\mu)$. Let $\varphi_\lambda$ and $\varphi_\mu$ denote the two actions of $\mu_{p^m}$ corresponding respectively to the integers $\lambda$ and $\mu$. Then a direct computation shows that we have the following commutative diagram: 
\[
\begin{tikzcd}
 & \mu_{p^m}\times \alpha_p \arrow{dr}{\varphi_\mu} \\
\mu_{p^m}\times \alpha_p \arrow{ur}{\psi_u \times \text{Id}} \arrow{rr}{\varphi_\lambda} && \alpha_p
\end{tikzcd}
\]
where $\psi_u \times \text{Id}$ is an isomorphism. Via the isomorphism $\psi_u \times \text{Id}$, we have that $\text{Ext}_\lambda (\mu_{p^m} , \alpha_p)\cong \text{Ext}_\mu (\mu_{p^m} , \alpha_p)\cong0$ (see for example Cor. 6.4 in \cite{DG70}). In particular, there is a unique extension which implies that $G_\lambda \cong G_\mu$. On the other hand, if we have such isomorphism it is straightforward to conclude after comparing the group laws of $G_\lambda$ and $G_\mu$. 
\end{proof}
\newpage
\section{Classification of the deformations of $G_\lambda$}
 In this section, we study the problem of classifying as explicitly as possible all the deformations of the groups $G_\lambda$. Deformation theory for general finite flat group schemes is a relatively complicated problem in the sense that the deformation functor is not always representable. In the case of deformations of commutative groups, the problem has been solved by Oort and Mumford in a positive way (see \cite{OM68}). For example, we mention now a useful result which will be used later on and it concerns the classification of finite flat $R$-group schemes of order a prime $p$. To be precise, Oort and Tate have proven the following (see \cite{OT70} or also \cite{Oor71}):
\begin{prop}\label{Tate}
Let $p$ be a prime. There exist a polynomial $W_p \in \mathbb{Z}_p [x, y]$ such that for all $R$ complete local Noetherian ring of residue characteristic $p$ and for any finite flat group scheme $G$ over $R$ of order $p$, we have that $G=\text{Spec}(A)$ where $A=R[\tau]/ (\tau^p -a\tau)$ with group operation given by: 
$$\tau \mapsto \tau_1+\tau_2 + b W_p (\tau_1 , \tau_2)$$
and $ab=p \in R$.
\end{prop}
\noindent
This result is a fundamental step in deformation theory of finite flat group schemes in order to understand how complicated deformations  might be in relation of possible necessary ramification of the base ring (here represented by the relation $ac=p$ in $R$). In particular, the above result describes all the possible deformations of the finite flat group scheme $\alpha_p$.
As the $k$-group schemes in the family $G_\lambda$ are non-commutative extra care is needed. We are ready now to present the first main result of the section: 
\begin{theorem}\label{maindef}
Let $p$ be a prime and let $R$ be a local Artin ring of perfect residue field $k$ of positive characteristic $p$. Let $G$ be a deformation over $R$ of the $k$-group scheme $G_\lambda \cong \alpha_p \rtimes_\lambda \mu_{p^m}$  with parameter $\lambda \in \{1, \cdots , p^{m}-1\}$.
Then the characteristic of $R$ is $p$ and we have the following classification:\\
\break
\indent(i) if $v_p (\lambda)\not=m-1$, all deformations of $G_\lambda$ over $R$ are trivial, i.e. 
$$G \cong (\alpha_p \rtimes_\lambda \mu_{p^m}) \times_{\text{Spec} (k)} \text{Spec}(R) ; $$
\indent(ii) if $v_p (\lambda)=m-1$, all deformations of $G_{p^{m-1}}$ over $R$ form a 1-dimensional family (over $k$) of\\
\indent\indent\indent\indent\indent\indent\indent non-commutative finite flat $R$-group schemes of order $p^{m+1}$.\\
\break
\indent In particular, in (ii), there is a unique $a\in k$ such that we have an isomorphism:
$$G \cong \tilde{H}_{a} \text{ as }R\text{-group schemes, }$$
\indent where $\tilde{H}_{a}\cong \text{Spec}(R[x, y] / (x^p , y^{p^m } ))$ is endowed with the group law:
$$
\begin{aligned}
x &\mapsto (1+y)^{p^{m-1}} \otimes x + x \otimes 1 +a\pi W_p (x\otimes 1 , 1\otimes x),\\
y &\mapsto y\otimes 1 + 1 \otimes y + y \otimes y.
\end{aligned}
$$
for the polynomial $W_p (x, x')=\frac{(x+x')^p -x^p -x'^p}{p}$. 
\end{theorem}
\noindent
We now proceed in proving the above result in several different steps.
The strategy is to proceed by induction on the lenght of $R$. When the lenght of $R$ is 1, everything follows by observing that one can take $R=k$. 
Now, assume that we are under the induction hypothesis. Let $\pi \in \text{Ann} (\mathfrak{m}_R )$ and let $R' := R/\pi R$. We have that  the length of $R'$ is strictly smaller than the one of $R$, so we can apply the inductive hypothesis. First, this implies that the characteristic of $R'$ is $p$, hence there exists an element $\gamma \in R$ such that $p=\gamma \pi$. Now, the deformation $G$ over $R$ of $G_\lambda$ can be described explicitly (scheme-theoretically) by polynomials $f$ and $g$ inside $R[x,y]$ such that, writing $G=\text{Spec}(A)$, 

$$A  \cong R[x, y ]/(x^p - \pi f (x, y) , y^{p^m} - \pi g(x, y )).$$
\noindent
Now, the Hopf $R$-algebra structure on $A$ is given by the group law: 

$$(x, y)\circ_G (x', y') = (x' (1+y)^\lambda +x +\pi h_1 (x, x', y, y') , y+y'+yy' +\pi h_2 (x, x', y, y')),$$
for certain polynomials $h_1$ and $h_2$ in $R[x, x' , y, y' ] / (x^p , x'^p , y^{p^m} , y'^{p^m} )$. \\
We first prove that the characteristic of $R$ is $p$, i.e. every deformation of $G_\lambda$ lives in positive characteristic $p$. Consider two generic points $h=(x,y)$ and $h' =(x', y')$ in $G$ and impose that $h\circ_{G} h'$ is still an element in $G$ to deduce relations concerning the polynomials $f, g, h_1$ and $h_2$. \\
The following conditions have to hold:
$$
\begin{cases}
 [ x' ( 1+y )^{\lambda} +x + \pi h_1 ]^p - \pi f (x' (1+y)^\lambda + x + \pi h_1 , y + y' + y y' +\pi h_2 )=0 \\
[y + y' + y y' +\pi h_2 ]^{p^m} - \pi g (x' (1+y)^\lambda +x+\pi h_1 , y + y' + y y' +\pi h_2 )=0
\end{cases}
$$
Since $\pi^2=0$ and since $R/\pi R$ has characteristic $p$ we have that: 
$$
\begin{cases}
 [ x' ( 1+y )^{\lambda} +x ]^p - \pi f (x' (1+y)^\lambda + x , y + y' + y y' )=0 \\
[y + y' + y y' ]^{p^m} - \pi g (x' (1+y)^\lambda +x , y + y' + y y' )=0
\end{cases}
$$
because $\pi f (z+\pi h_1 , z' + \pi h_2)=\pi f (z, z')$. Now, since $g$ and $g'$ belong to $G$ we have that $x^p-\pi f(x, y) = {x'}^p -\pi f(x' , y') =0$ and focusing on the first equation in the above system we have: 
$$
{x'}^p (1+y)^{\lambda p} + \sum_{k=1}^{p-1} {{p}\choose{k}}[x' (1+y)^\lambda]^kx^{p-k} + x^p -\pi f (x' (1+y)^\lambda +x , y+y'+yy')=0
$$
so we deduce that: 
$$
\pi f(x', y') (1+y)^{\lambda p} + [x' (1+y)^\lambda +x ]^p - {x'}^p (1+y)^{\lambda p} -x^p + \pi f (x, y) -\pi f (x' (1+y)^\lambda +x , y+y'+yy')=0.
$$
Denoting by $W_p (x, x')= \frac{(x+x')^p -x^p -{x'}^p}{p}$, the expression becomes:
$$
\pi  [\gamma W_p (x, x' (1+y)^\lambda )- f(x' (1+y)^\lambda +x , y+y+yy') + f(x,y) + f(x', y') (1+y)^{\lambda p}]=0
$$
imposing $y=y'=0$, we get:
$$
\pi [\gamma W_p (x, x') - f(x+x' , 0) + f(x , 0) + f(x' , 0)]=0
$$
Note that inside the square parenthesis, there is the monomial $\gamma x' x^{p-1}$ (coming from the polynomial $W_p$) and it is the only monomial of that type because $\text{deg}_x (f(x, y)) \leq p-1$. We deduce that $\gamma \in \mathfrak{m}_R$ and as a consequence, since $\pi\gamma=p$ and since $\pi \in \text{Ann}(\mathfrak{m}_R )$ we deduce that $p=0$ in $R$, i.e. $\text{char}(R)=p >0$. This completes the first part of the proof. 
\begin{remark}\normalfont
It is interesting to notice that the above computation still holds when $\lambda=0$, i.e. if we were in a situation where $G$ is a deformation of the $k$-group scheme $G_0 \cong \alpha_p \times \mu_{p^m}$. However, it is clear that deformations of $\alpha_p \times \mu_{p^m}$ exists also over some ring characteristic 0, e.g. $H \times \mu_{p^m}$ where $H$ is a deformation of $\alpha_p$ (described by Oort and Tate in \cite{OT70}) and all groups are taken over a ring $R$ over which $p$ ramifies. Indeed, the issue is in the assumption that a generic deformation $G$ are only of the form considered in this article. This fact is indeed false when $\lambda=0$. 
\end{remark}
\noindent
Taking up again the system of equations of before, and implementing the new information that the characteristic of $R$ is $p$, we have that: 

$$
\begin{cases}
f (x' (1+y)^\lambda +x , y+y'+yy') = f(x,y)+f(x', y') (1+y)^{\lambda p} \\
g(x' (1+y)^\lambda +x , y+y'+yy')=g(x,y)+g(x', y')
\end{cases}
$$
for the second equation we have used that:

$$(y+y'+yy')^{p^m} - \pi g(x' (1+y)^\lambda +x , y+y'+yy')=0$$
and so
$$y^{p^m} + {y'}^{p^m}+y^{p^m} {y'}^{p^m} -\pi g (x' (1+y)^\lambda +x, y+y'+yy')=0$$
and substituting inside
$$y^{p^m}-\pi g(x,y)={y'}^{p^m} -\pi g(x', y')=0.$$
In order to understand the underlying $R$-scheme structure of $G$, we have to classify all the possible polynomials $f, g \in R[x, y]$, knowing that the characteristic of $R$ is $p$. It might be possible that a direct computational approach allows one to classify $f$ and $g$ by imposing that the couple $f, g$ has to satisfy that if $h$ and $h'$ are two generic points in $G$ then the coordinates of $h\circ_G h'$ have also to satisfy the equations $x^p -\pi f(x, y) = y^{p^m} - \pi g(x, y)$. We prefer here to proceed in a more conceptual way. Before going deeper in such approach we take a little detour to introduce cohomology of linear representation of finite flat group schemes (in terms of Hochschild cohomology) which will play a central role in what follows. 
Consider the following general situation. Let $R$ be a sufficiently nice ring, e.g. in order to cover all the situations we have to deal with it is enough to take the class of "models" in the sense of Demazure and Gabriel (see the beginning of \cite{DG70}). These rings form a full subcategory of the category of rings which is exactly the subcategory of $U$-small rings where $U$ is a certain Grothendieck universe.\\
Let $H$ be a finite, flat group scheme over $R$ and let $V$ be a projective $R$-module of finite type. Denote by $\text{GL} (V)$ be the $R$-group functor sending each $R$-algebra $S$ to $\text{Gl} (V \otimes_R S)$ (where $\text{Gl}$ is the usual general linear group). A linear representation of $G$ in $V$ is a natural transformation (i.e. a morphism of functors) $\rho : G \rightarrow \text{GL} (V)$. Equivalently, writing $G= \text{Spec}(A)$ for the $R$-Hopf algebra $A$, the morphism $\rho$ corresponds to a $R$-linear map $\Delta_\rho : V \rightarrow V \otimes_R A $ such that for all $v\in V$ we have that $\Delta_\rho (v) = \rho (g_0 )v_A \in V\otimes_R A $ (here $v_A =v \otimes_R 1_A $) where $g_0 \in G(A)$ is the element corresponding to the identity map on $A$. Of course, the map $\Delta_\rho$ also reflects other functorial relations corresponding to the fact that $A$ is a $R$-Hopf algebra. The point of view of considering $\Delta_\rho$ instead of $\rho$ is useful as it allows explicit computations as we will see later on. Now, starting from an $R$-module $V$, we denote by $V_a$ the commutative $R$-group functor on the category of $R$-algebras sending $S \mapsto V_a (S)=V\otimes_R S$.\\
For all positive integers $n$, we define the $n$-cohomology group $H^n (G, V) := H^n_{\text{Hoc}} (G, V_a )$ where $H^*_{\text{Hoc}} (G, V_a )$ denotes the Hochscild cohomology of the $G$-module $V_a$. An important point is that the groups $H^n (G , V)$ depend only on the action of $G$ on $V$. To be precise, there exists a complex $C^{*} (G, V)$ where $C^n (G, V)= V \otimes_R A \dots \otimes_R A$ ($n$-times $\otimes_R A$) whose boundary maps depend only on $\Delta_\rho$ and such that $H^n (G, V )\cong H^n (C^n (G, V))$. For more information on this, we refer the reader to the book of Demazur and Gabriel (see section 3 Chapt. 2 in \cite{DG70}).\\
Another tool we need for studying deformation of non-commutative schemes is the adjoint representation of a finite flat group scheme: As we will see later on, such special representation plays a central role in deforming group structures. 
Let $R(\epsilon)$ be the $R$-algebra of dual $R$-numbers, i.e. $R[T]/( T^2 )$. Denote by $p: R(\epsilon) \rightarrow R$ the projection which sends $p(1)=1$ and $p(\epsilon)=0$. The Lie group of a finite flat group scheme $G$ over $R$ is defined as $\text{Lie}(G) (R) = \text{Ker} ( p: G(R(\epsilon))\rightarrow G(R))$. The group $G$ acts functorially on $\text{Lie}(G)$ in the following way, called adjoint representation of $G$: 

\begin{align*}
 \text{Ad}_{G} : G(R) &\rightarrow \text{GL} (\text{Lie}(G) (R))\\
    g &\mapsto (x \mapsto i(g) x i(g)^{-1} )
\end{align*}
where $i: G(R) \rightarrow G(R(\epsilon))$ is induced by the injection $i : R \hookrightarrow R(\epsilon)$ such that $i(1)=1$. \\
We recall that there are two compatible operation on $\text{Lie}(G)(R)$ by $R$ and by $G(R)$. Now, we have a canonical isomorphism as $R$-modules (see for example section 4 Chapt. II in \cite{DG70}) $\text{Lie} (G)(R) \cong D_a (\omega_{G\otimes k}) (R) \cong \text{Mod}_k (\omega_{G\otimes k} , R )$. Moreover, since $G$ is finite flat we have that such $R$-module is also canonically isomorphic to $\text{Mod}_k (I_G / I_G^2 , R)$ where $I_G$ is the augmentation ideal of $G$. Using the additive notation for the elements of the Lie group of $G$ (namely $x=e^{\epsilon x}$), we have that the explicit adjoint representation as $R$-module: 
$$
 \text{Ad}_{G} : G(R) \rightarrow \text{GL} (\text{Mod}_k (I_G / I_G^2 , R))
$$
is determined by the formula $g e^{\epsilon f} g^{-1}= e^{\epsilon Ad_G (g) f}$ for all $g\in G(R)$ and all $f\in \text{Mod}_k (I_G / I_G^2 , R)$. We denote with $V_{\text{ad}} = \text{Mod}_k (I_{G_\lambda} / I_{G_\lambda}^2 , R)$ the adjoint representation of $G_\lambda$.
Finally, we are ready continue our study of the scheme structure of the deformations $G$ over $R$ of $G_\lambda$.
Because of what has been done previously, we know that:
$$G=\text{Spec} (R[x, y] / (x^p -\pi f(x, y)) , y^{p^m} -\pi g(x, y))$$
\noindent
with group law depending on the parameter $\lambda \in \{ 1, \dots, p^m -1\}$ and certain polynomials $h_1$ and $h_2$ in $R[x, x', y, y']$.
Define the vector $F(x, y) = (f(x, y) , g(x, y))^T$ and $\overline{F} (x, y ) :=  F(x, y) \mod \mathfrak{m}_R$. Consider now as usual two generic elements $h=(x, y)$ and $h'=(x', y')$ in $G_\lambda$, we have:

$$\overline{F}(h\circ_{G} h' )=\overline{F}((x,y)\circ_{G_\lambda} (x' , y'))=
    \begin{pmatrix}f((x,y)\circ_{G_\lambda} (x' , y'))\\
    g((x,y)\circ_{G_\lambda} (x' , y'))
    \end{pmatrix}
    \text{ mod } \mathfrak{m}_R
    =
    \begin{pmatrix}f(x,y)+f(x' , y') (1+y)^{\lambda p}\\
    g(x,y) + g(x' , y')
    \end{pmatrix}\text{ mod } \mathfrak{m}_R
    =$$
    $$=
    \begin{pmatrix}f(x,y)\\
    g(x,y)
    \end{pmatrix}
    +
    \begin{pmatrix}(1+y)^{\lambda p}&0\\
    0&1
    \end{pmatrix}
    \begin{pmatrix}f(x',y')\\
    g(x',y')
    \end{pmatrix}\text{ mod } \mathfrak{m}_R
    $$
    hence we deduce the relation:
    $$\overline{F}(h\circ_{G_\lambda} h' )=\overline{F}(h)+\begin{pmatrix}(1+y)^{\lambda p}&0\\
    0&1
    \end{pmatrix}\overline{F}(h' )$$
    \noindent
Define now the $k$-vector space $\overline{V} =k e_1 \oplus k e_2$ over which $G_\lambda$ acts via:
$$
\Delta : \overline{V} \rightarrow \overline{V} \otimes_k \faktor{k[x, y]}{(x^p , y^{p^m} )}
$$
where $\Delta (e_1 ) =  e_1 \otimes (1+y)^{\lambda p}$ and $\Delta (e_2 )= e_2$. \\
The relation proven above for $\overline{F} (x, y)$ is equivalent to say that $\overline{F}$ is a crossed homomorphism from $G_\lambda$ to $\overline{V}$ and as such, we can identify it with a 1-cocycle for the representation $\overline{V}$, i.e. $\overline{F} \in H^1 (G_\lambda , \overline{V})$.\\
We recall now that $G_\lambda \cong \alpha_p \rtimes_\lambda \mu_{p^m}$, i.e. it fits into a split exact sequence $0 \rightarrow \alpha_p \rightarrow G_\lambda \rightarrow \mu_{p^m} \rightarrow 0$. The action of $G_\lambda$ on $V$ factors via the quotient $\mu_{p^m}$. 
Now, diagonalizable group schemes, such as $\mu_{p^m}$, satisfy the following useful property (on which we rely also later on, so we state it properly): 
\begin{prop}\label{diag}
Let $H$ be a diagonalizable $k$-group scheme and $\rho: H \rightarrow \text{GL} (V)$ a linear representation of $H$ then for all $n>0$ we have that $H^n (H, V)=0$. 
\end{prop} 
\begin{proof}
See for example Prop. 4.2, Sec. 3, Chapt. II in \cite{DG70}. This result holds for any base $R$ which is a "model" in the sense of \cite{DG70}. 
\end{proof}
\noindent
Moreover, we can combine the above result with the following:
\begin{prop}\label{quot}
    Let $G$ be a finite flat $k$-group scheme and let $N$ be a finite flat normal subgroup scheme inside $G$. Let $W$ be any $G$-module and assume that $G/N$ (which is as well a finite flat $k$-group scheme) is diagonalizable. Then for all $i\geq 0$ we have isomorphisms:
    $$H^i (G , W) \xrightarrow{\cong} H^i (N , W )^{G/N}.$$
\end{prop}

\begin{proof}
This is a consequence of proposition \ref{diag} applied to Grothendieck's generalization of the Lyndon-Hochschild-Serre spectral sequence. See for example Corollary 6.9 in \cite{Jan87}.
\end{proof}
\noindent
We recall that the action of $G$ on the cohomology groups $H^* ( N , W )$ is deduced by the action of $G$ on the cocycles $C^* (N , W ) = W\otimes \mathcal{O}(N)^{\otimes *}$ which comes from the action of $G$ on $N$ via conjugation and on $W$ via its $G$-module structure (see for example section 6.7 in \cite{Jan87}). 
Finally, we have that Proposition \ref{quot} allows us to compute explicitly:
\begin{prop}
Let $\lambda \in \{1, \cdots , p^{m}-1\}$, we have 
$H^1 (G_\lambda , \overline{V}) =0$.
\end{prop}
\begin{proof}
First, we recall that the $G_\lambda$-representation $\overline{V}$ decomposes as a direct sum $L\oplus \mathbb{G}_a$ (with $L\cong \mathbb{G}_a$ as $k$-group schemes) where the action of a generic element $g=(x, y)\in G_\lambda$ on $L$ is given by the multiplication by $(1+y)^{\lambda p}$ and it is the trivial one on the second factor $\mathbb{G}_a$. By Proposition \ref{quot}, we have that:
$$H^1 (G_\lambda , \overline{V})\cong H^1 (G_\lambda , L) \oplus H^1 (G_\lambda, \mathbb{G}_a )\cong H^1 (\alpha_p , L )^{\mu_{p^m}} \oplus H^1 (\alpha_p , \mathbb{G}_a )^{\mu_{p^m}}.$$
We have that $H^1 (\mathbb{G}_a , \mathbb{G}_a )\cong \text{Gr}_k (\mathbb{G}_a , 
\mathbb{G}_a )$ is the $k$-group scheme of endomorphisms of $\mathbb{G}_a$ which is isomorphic to the group $\{p(x)=\sum_{k\geq 0} 
a_i x^{p^i} \text{ with }a_i\in k \}$ with the usual addition of polynomials (see for example sec. 3.4 in Chapter 2 in 
\cite{DG70}). Hence, we have that any 1-cocycle $c: \alpha_p \rightarrow \mathbb{G}_a$ corresponds to a polynomial $p_c (x)=a_c x$ for some $a_c \in k$. A direct computation shows that both actions of $\mu_{p^m}$ of $L$ and the trivial one on $\mathbb{G}_a$ are incompatible with the conjugation on $N$ (which corresponds to sending a generic element $x$ of $\alpha_p$ in $(1+y)^\lambda x$). As a consequence, we have that $H^1 (\alpha_p , L )^{\mu_{p^m}} = H^1 (\alpha_p , \mathbb{G}_a )^{\mu_{p^m}}=0$.
\end{proof}
\noindent
We deduce then that $\overline{F}$ is a 1-cobord, or equivalently, there exists $(c, d) \in k^2$ such that for every $h=(x, y) \in G_\lambda$ we have that 
$\overline{F} (h) = \rho (h) \cdot (c, d)^T - (c, d)^T$. After making the representation $\rho$ explicit, we have that there exists some polynomials $z_1$ and $z_2$ in $R[x, y]$ with coefficients in $\mathfrak{m}_R$ such that
$$F(x, y) = \begin{pmatrix} f(x, y) \\ g(x, y) \end{pmatrix}= \begin{pmatrix} c[(1+y)^{\lambda p}-1] + z_1 \\ z_2 \end{pmatrix}.$$
\noindent
Since $\pi \in \text{Ann} (\mathfrak{m}_R)$, we deduce that $\pi f(x, y) = c \pi [(1+y)^{\lambda p}-1]$ and $\pi g(x, y) =0$. Finally, we conclude that the underlying $R$-scheme structure of the generic deformation $G$ over $R$ of $G_\lambda$ has to have the form:
$$G=\text{Spec}\Big(\faktor{R[x, y]}{(x^p -c\pi [(1+y)^{\lambda p}-1] , y^{p^m} )}\Big).$$
Note also that (since $\text{char}(R)=p$) $(1+y)^{\lambda p} = 1+y^{\lambda p}$ if and only if $\lambda$ is a power of $p$. Hence, we first simplify a bit the expression $c\pi [(1+y)^{\lambda p} - 1]$ by using the results in section 2. By Lemma \ref{units}, we have indeed that without loss of generality we can assume that $\lambda=p^v$ for a certain integer $v\geq 0$ and so $c\pi [(1+y)^{\lambda p}-1]=c\pi y^{p^{v+1}}$.
Summarizing what we have proven until now, we have that for every deformation $G$ of $G_\lambda$ we have the following isomorphisms of $R$-schemes:
$$ G\cong\text{Spec}\Big(\faktor{R[x, y]}{(x^p -c\pi y^{p^{v+1}} , y^{p^m} )}\Big) \indent\;\text{ for some }c\in k $$
\noindent
Our goal is to prove the following stronger description: 
\begin{prop}
For all $\lambda\in \{1, \cdots , p^{m}-1\}$, we have that every finite flat deformation $G$ over $R$ of $G_\lambda$ is isomorphic as $R$-scheme to $\text{Spec}\Big(\faktor{R[x, y]}{(x^p , y^{p^m} )}\Big)$, i.e. to the base change $G_\lambda \otimes_k R$.
\end{prop}
\noindent
Our aim is now to prove that as $R$-schemes, we have that $G$ is isomorphic to $ \text{Spec}(R[x, y ] / (x^p , y^{p^m}))$, i.e. $c=0$. This point of the proof is the most delicate as the situation is substantially different  from the one treated in \cite{Sch01} where $v_p (\lambda)=0$ and a new approach is necessary. Indeed, if $v_{p} (\lambda) =0$, it is possible to directly conclude by a linear isomorphic substitution after observing that $x^p -c\pi[(1+y)^{\lambda p} -1 ] = (x-\omega y)^p$ after suitably extending $R$ finitely flat by adding the element $\omega$ such that $\omega^p=c\pi$. However, when $v_p (\lambda) >1$, the situation is much more subtle because such substitution is not available.
Our first goal is to determine the possible values of the constant $c\in k$ and the possible polynomials $h_1$ and $h_2$ up to automorphisms. Since $G$ is a $R$-group scheme with neutral element $(0, 0)$ we have that in order for the group axioms to be satisfied we have that $h_i (x,y,0,0) = h_i (0,0,x' ,y' )=0$ for $i=1, 2$. The strategy is to use the inductive hypothesis on $R/\pi R$ and $k$ to transport certain group structures from the infinitesimal deformations (i.e. deformations over $R/\pi R$ where $\pi^2=0$) to $G$ over $R$. 
In order to deal with infinitesimal deformations, we rely on two results. The first one deals with deformations of group laws (this is a specialization of Theorem 3.5 in \cite{SGA3} SGA 3 Exp. III): 

\begin{theorem}\label{group}
Let $S$ be a scheme and let $I$ and $J$ be two quasi-coherent ideals such that $J \subset I $ and $I \cdot J=0$ defining respectively closed sub-schemes $S_0$ and $S_J$.\\
Let $X$ be a finite flat $S$-scheme and denote by $X_J$ and $X_0$ the $S$-subschemes of $X$ obtained by base change via natural projections modulo the ideals $I$ and $J$. Assume that $X_J$ has the structure of $S$-group scheme and denote by $L_0$ the commutative $S_0$-group functor given by the derivations of $X_0 /S_0$, i.e. 
$$\text{Hom}_{S_0} (* , L_0) :=\text{Hom}_{\mathcal{O}_{*}} 
(\omega^1_{X_0 / S_0} \otimes_{\mathcal{O}_{S_0}} \mathcal{O}_{*}, J \otimes_{\mathcal{O}_{S_0}} \mathcal{O}_{*}).$$
The $S_0$-group functor $L_0$ acts on $X_0$ via its adjoint representation. Moreover, the existence of a structure of $S$-group scheme on $X$ is equivalent to the following two conditions:\\
\break
\indent $(i)$ there exists a $S$-scheme morphism $P : X \times X \rightarrow X$ which induces modulo $J$ the group law \\\indent\indent$P_J $ of $X_J$,\\
\indent $(ii)$ a certain obstruction class $c(P_J) \in H^3 (X_0 , L_0 )$ (corresponding to the associativity property which  \\\indent\indent$P$ has to satisfy) is zero. \\
\break
In addition, if the conditions $(i) , (ii)$ are both satisfied, the set $E$ of group laws of $X$ (modulo $S$-automorphisms of $X$) inducing the group law $P_J$ on $X_J$ is a principal homogeneous space for the abelian group $H^2 (X_0 , L_0 )$. 
\end{theorem}
\noindent
The proof of this result is particularly useful as it allows explicit computations. Indeed,  we want to use this result to impose explicit conditions on the possible group laws on $G$, which amount to imposing conditions on the polynomials $h_1$ and $h_2$. Let $E$ be the set of $R$-automorphisms orbits of group laws $P$ on $G$ such that after reducing modulo $(\pi)$ we have that $P$ coincides exactly with the group law of $G_\lambda$ over $R/\pi R$ (by inductive hypothesis). By the above result, we know that $E$ is a principal homogeneous space for the abelian group $H^2 (G_\lambda /k , V_{\text{Ad}})$ where $V_{\text{Ad}}$ is the $k$-adjoint representation of $G_\lambda /k$. Note that for this fact we are using that the groups involved are finite and flat. 
Denote the action of an element $\delta \in H^2 (G_\lambda /k , V_{\text{Ad}})$ on an element $\mu\in E$ by $\delta\cdot \mu$. Now, in order to make things explicit we consider the action of $H^2 (G_\lambda /k , V_{\text{Ad}})$ on $E$ on points. Set a coefficient ring, say the $R$-algebra $S$. We have that for all $g, g' \in G (S)$ and for all $\mu , \mu' \in E(S)$ we have that:
$$\mu' (g, g' ) = \delta (\mu , \mu' )(g_0 , g_0') \cdot \mu (g, g') = \mu (g, g') + D_{\delta(\mu , \mu') (g_0 , g_0' )} \in G(S)=\text{Hom}_{R\text{-Alg}} (A, S)$$
where the derivation $D_{\delta(\mu, \mu')(g_0 , g_0' )}$ is the derivation associated to the image of $(g_0 , g_0' )$ via the 2-cocycle $\delta(\mu , \mu')\in H^2 (G_\lambda , V_{Ad})$ 
which is the one corresponding to the couple $(\mu, \mu')$ (note that we are using that $E$ is a principal homogeneous space, i.e. the action of the abelian group $H^2 (G_\lambda , 
V_{Ad}))$ is free and transitive which grant the existence and unicity of $\delta$ in function of $\mu$ and $\mu'$). Note also that as $k$-module homomorphism, the image of 
$D_{\delta(\mu , \mu') (g_0 , g_0' )}$ is contained in the ideal $(\pi)$ in $R$. This is indeed a fundamental point which ensures that the $R$-algebra homomorphism determined by $\mu (g, g')$ plus the 
$R$-module homomorphism $D_{\delta(\mu , \mu') (g_0 , g_0' )} $ is still an $R$-algebra homomorphism. \\
Another result which allows us to transport group structures via infinitesimal deformations concerns lifts of group morphisms.
To be precise, we have the following (this is specialization of Theorem 2.1 in \cite{SGA3}, SGA 3 Exp. III):
\begin{theorem}\label{groupm}
Let $S$ be a scheme and let $I$ and $J$ be two quasi-coherent ideals such that $J \subset I $ and $I \cdot J=0$ defining respectively closed sub-schemes $S_0$ and $S_J$.\\
Consider the following:\\
\break
\indent (i) $X$ an $S$-group scheme,\\
\indent (ii) $L_0$ the commutative $S_0$-group scheme given by the derivations of $X_0 /S_0$, i.e. \\
$$\text{Hom}_{S_0} (* , L_0) :=\text{Hom}_{\mathcal{O}_{*}} 
(\omega^1_{X_0 / S_0} \otimes_{\mathcal{O}_{S_0}} \mathcal{O}_{*}, J \otimes_{\mathcal{O}_{S_0}} \mathcal{O}_{*}),$$\\
\indent(iii) $Y$ a flat $S$-group scheme and $f_J : Y_J \rightarrow X_J$ a morphism of $S_J$-group schemes.\\
\break
Then we have that $f_J$ lifts to a $S$-group scheme morphism $f: Y \rightarrow X$ if and only the following two statements hold:\\
\break
\indent(1) $f_J$ lifts to a $S$-scheme morphism $f: Y \rightarrow X$,\\
\indent(2) a certain obstruction class $c(f_J) \in H^2 (Y_0 , L_0 )$ is zero.
\end{theorem}
\noindent
As explained before, using the additive notation for the elements of the Lie group of $G$ (namely $x=e^{\epsilon x}$), we have that the explicit adjoint representation as $R$-module: 
$$
 \text{Ad}_{G} : G(R) \rightarrow \text{GL} (\text{Mod}_k (I_G / I_G^2 , R))
$$
is determined by the formula $g e^{\epsilon f} g^{-1}= e^{\epsilon Ad_G (g) f}$ for all $g\in G(R)$ and all $f\in \text{Mod}_k (I_G / I_G^2 , R)$. Let now consider the case where $G=G_\lambda$ defined over a finite perfect field $k$ of positive characteristic $p$. We know that:
$$
G_\lambda = \text{Spec}\Big(\faktor{k[x, y]}{ ( x^p , y^{p^m} )}\Big)
$$
with group law given by
$(x, y) \circ_{G_\lambda} (x' , y' )= (x' (1+y)^\lambda +x , y+y'+yy')$. \\
Let $g=(x' , y' )$ and denote by $g^{-1} = ( x'' , y'' )$. The following relations hold:

$$
\begin{cases}
(1+ y' )^{\lambda} x'' + x' =0 \\
y' + y'' + y' y'' =0
\end{cases}
$$
Let $f(x,y)= a_f x + b_f y \in \text{Mod}_k (I_{G_\lambda} / I_{G_\lambda}^2 , R)$ where $I_{G_\lambda}=(x,y)$ as an ideal inside $\mathcal{O}(G_\lambda)$. Now, we impose that $g e^{\epsilon f} g^{-1}= e^{\epsilon Ad_G (g) f}$. In other words, we have to compute $g e^{\epsilon f} g^{-1} = f(g \circ (x,y) \circ g^{-1})$ as an element in $\text{Mod}_k (I_{G_\lambda} / I_{G_\lambda}^2 , R)$. We have that: 

$$f( (x' , y' ) \circ (x,y) \circ (x'' , y'' ))=
f((1+y')^{\lambda} [(1+y)^\lambda x'' + x] +x' , y' + y+y'' +yy'' +y' (y+y'' + yy'') ) )$$
Now, using the relations between $g$ and $g^{-1}$ and that we are doing computations modulo the ideal $(x, y)^2$ inside $\mathcal{O}(G_{\lambda})= k[x,y] (x^p , y^{p^m})$, we conclude that the above expression is equal to:
$$f((1+y' )^{\lambda} x -\lambda x' y , y)= f\Bigg(\begin{pmatrix}
        (1+y' )^\lambda & -\lambda x' \\ 0 & 1
    \end{pmatrix}\begin{pmatrix}
         x \\  y
    \end{pmatrix}\Bigg)$$
where we used that $(1+y)^\lambda = 1+\lambda y \mod (x,y)^2$. We deduce finally that the a adjoint representation of $G_\lambda$ is the representation: 

\begin{align*}
 \text{Ad}_{G_\lambda} : G_\lambda (R) &\rightarrow \text{GL} (\text{Mod}_k (I_{G_\lambda} / I_{G_\lambda}^2 , R))\\
    g=(x , y ) &\mapsto \begin{pmatrix}
        (1+y)^\lambda & -\lambda x \\ 0 & 1
    \end{pmatrix}
\end{align*}
We have now enough information for computing the cohomology of $V_{\text{ad}}$ with respect to the $G_\lambda$-action. 
To be precise, we have the following:
\begin{prop}\label{adjoint}
Let $\lambda\in \{1, \cdots, p^{m}-1\}$ and let $V_{ad}$ be the adjoint representation of $G_\lambda$ over $k$. \\
We have that:
$$H^2 (G_\lambda , V_{\text{ad}})=
\begin{cases}
\langle W_p\rangle \indent\indent\indent \;\;\text{ if }\; v_p (\lambda)=m-1, \\
0 \indent\indent\indent\indent\;\;\; \text{ otherwise};
\end{cases}$$
where the class $W_p$ represents the polynomial $W_p (x, x' )= \frac{1}{p}[(x+x' )^p -x^p -x'^p].$
\end{prop}
\begin{proof}
When $v_p (\lambda)=0$, the claim has been proven by Schoof (see Lemma 3.2 in \cite{Sch01}). Assume now that $v_p (\lambda) \geq 1$. 
It is well-known that $H^2 (\alpha_p , \mathbb{G}_a )\cong \langle W_p \rangle$ (see for example Cor. 4.8 in section 3 in chapter II in \cite{DG70}). By Proposition \ref{quot}, we have that:
$$H^2 (G_\lambda , V_{\text{ad}}) \cong H^2 (\alpha_p , V_{\text{ad}})^{\mu_{p^m}} \cong H^2 (\alpha_p , L' )^{\mu_{p^m}}\oplus H^2 (\alpha_p , \mathbb{G}_a)^{\mu_{p^m}},$$
where $L' \cong \mathbb{G}_a$ (as k-group schemes) has a $\mu_{p^m}$-module structure given by the action of $\mu_{p^m}$ sending $x$ to $(1+y)^\lambda x$ and where the action of $\mu_{p^m}$ on $\mathbb{G}_a$ is the trivial one. It is straightforward to check how $\mu_{p^m}$ acts on every 2-cocycle $\alpha_p^2 \rightarrow \mathbb{G}_a$
by directly checking how it acts on the 2-cocycle given by $W_p$. In order to find the $\mu_{p^m}$-invariants, we have to impose that the 2-cocycle $W_p$ is invariant under the simultaneous action of $\mu_{p^m}$ via conjugation on $\alpha_p$ and the adjoint representation action on either the first factor $L'$ or the second factor $\mathbb{G}_a$. A direct computation shows that only one case is non-trivial. Indeed, let $\gamma \in 
\mu_{p^m}$ represented by the variable $y$. Then considering the trivial action of $\mu_{p^m}$ on $\mathbb{G}_a$ we have: 
$$W_p (\gamma (x, x' ))= W_p ((1+y)^\lambda x , (1+y)^\lambda x' )= (1+y)^{\lambda p} W_p (x, x' ).$$
Hence, we have that $H^2 (\alpha_p , \mathbb{G}_a )^{\mu_{p^m}}$ is different from zero if and only if $$W_p (\gamma (x, x' ))=W_p (x, x')$$ which happens, by the above computation, if and only if $v_p (\lambda)=m-1 $ (note that we are assuming $\lambda \geq 1$). We deduce that $H^2 (\alpha_p , L' )^{\mu_{p^m}}=0$ and $H^2 (\alpha_p , \mathbb{G}_a)^{\mu_{p^m}}\cong \langle W_p \rangle$ (if and only if $v_p (\lambda)=m-1$).
\end{proof}
\noindent
As a consequence of proposition \ref{adjoint}, we can completely classify all group structures of the generic deformation of $G_\lambda$ by describing explicitly the polynomials $h_1$ and $h_2$. To be precise, we have the following:

\begin{lemma}\label{lemmah}
The polynomials $h_1$ and $h_2$ in $R[x, y, x' , y' ]$ satisfy the following properties:
$$\text{ if }v_p (\lambda)\not=m-1, \text{ then }\;\; h_1=h_2=0,$$
$$\indent\indent \text{ if }v_p (\lambda)=m-1, \text{ then }\;\; h_1\in \langle W_p (x, x' )\rangle \text{ and }h_2 =0.$$
\end{lemma}
\begin{proof}
We have that the action of  $H^2 (G_\lambda , V_{Ad})$ on the principal homogeneous space $E$ of group laws on $G$ depends only on the derivations evaluated in $\alpha_p$. A direct computation shows that for $g=(x, y)$ and $g'=(x' , y' )$ the morphism $\mu_0 (g, g' ):= (x+x' , y + y' + yy')$ defines a group law on $G$, i.e. $\mu_0 \in E$, and because of Theorem \ref{group} any other group law $\mu$ is uniquely obtained by translations via derivations, i.e. $\mu(g, g' )=\mu_0 (g, g' ) + D_{\mu_0 , \mu} (g_0 , g'_0)$ where $g_0$ and $g'_0$ are the reductions mod $k$ of $g$ and $g'$ and are generic elements of $G_\lambda$. Finally, since the derivation $D_{\mu_0 , \mu} (g_0 , g'_0)$ has image living in the ideal $(\pi)$ and depends only on $x$ and $x'$, in the only non-zero case (i.e. when $v_p (\lambda)=m-1$) we have that the image of $D_{\mu_0 , \mu} (g_0 , g'_0)$ is exactly $a\pi W_p (x, x')$ for some $a\in k$. 

\end{proof}
\noindent
Note that in all cases we have $h_i ( x ,0)=h_i (0, x' )=0$ for $i=1, 2$ because of the axiom for the neutral element. We recall that now we restrict our attention to the case $v_p (\lambda) \not=m-1$. Consider the $R$-scheme morphism: 
$$\varphi : N:=\text{Spec} (R[x] / (x^p ))\hookrightarrow G=\text{Spec}(R[x, y] / (x^p -c\pi y^{p^{v+1}} , y^{p^m}))$$ 
where $\varphi$ corresponds to the natural projection for $y=0$. 
\begin{lemma}
For all $\lambda \in \{1, \cdots, p^m -1\}$, the $R$-scheme $N$ has a group law for which it is a normal $R$-subgroup scheme of $G$ with the closed immersion given by the $R$-group morphism $\varphi$ and $N\cong\alpha_p$ as $R$-group schemes. 
\end{lemma}

\begin{proof}
By inductive hypothesis, we know that after taking a base change to $R/\pi R$, we have that $H\otimes R/\pi R$ has the structure of an $R/\pi R$-group scheme isomorphic to $\mu_{p^{m}} /R/\pi R$ and $\varphi \otimes R/\pi R$ is a $R/\pi R$-group scheme homomorphism. The idea now is to use the results from SGA mentioned above to transport these group structures from the infinitesimal deformations over $R/\pi R$ to $R$. Note the fundamental fact that $N$ is a flat $R$-scheme of finite type over $R$.\\
According to Theorem \ref{group}, in order to prove that $N$ has the structure of group scheme we have to verify two statements. First that there exists a $R$-scheme morphism $P_N : N \times N \rightarrow N$ which specialize to the group law of $\alpha_p$ after base change to $R/\pi R$ and second, that a certain class $c(P_N ) \in H^{3} (\alpha_p /k , V_{\text{Ad}})$ (corresponding to the associativity property of $P_N$) is zero. For the first part, note that $G$ as a $R$-group scheme has its group law which can be seen as a $R$-scheme morphism $P: G \times G \rightarrow G$ which induces the usual group law on $G_\lambda$ after base changing to $R/\pi R$ (by inductive hypothesis). Because of Lemma \ref{lemmah}, taking the restriction of $P$ to generic elements of $N$ we get that $P((x, 0), (x' , 0 ))= (x + x' +\pi h_1 , 0)$. Hence, defining $P_N$ as the restriction of $P$ to $N$ we get a well-defined $R$-schemes morphism $P_N : N \times N \rightarrow N$ satisfying the required properties. Now, the class $c(P_N ) \in H^3 (\alpha_p , V_{\text{Ad}})$ which corresponds to the associativity property of $P_N$ is nothing else than the image of the respective class $c(P) \in H^3 (G_\lambda , V_{\text{Ad}})$ under the group homomorphism $H^3 (G_\lambda , V_{\text{Ad}}) \rightarrow H^3 (\alpha_p , V_{\text{Ad}})$ induced by the inclusion $\alpha_p \hookrightarrow G_\lambda$ (because $P_N$ is defined as the restriction of $P$ to $N\times N$). Since $P$ is a group law, we have that $c(P)=0$ which implies that $c(P_N )=0$. We conclude that $N$ has the structure of a finite flat group $R$-scheme. Now, applying Theorem \ref{groupm} to the $R$-scheme morphism $\varphi$ together with the (necessary) induction hypothesis that $\varphi \otimes R/\pi R$ is the closed immersion corresponding to the projection $y=0$, we can conclude that $\varphi$ is also a closed immersion $R$-group scheme morphism. \\
Alternatively, it can be checked by formulas that the $R$-scheme morphism $\varphi$ preserves the group law via a direct computation as everything is explicit. Finally, because the group law of $N$ is explicit, another direct computation shows that $N$ is a normal $R$-subgroup scheme inside $G$. This concludes the lemma. 
\end{proof}
\noindent
Now that we endowed $N$ with a finite flat $R$-subgroup scheme structure which makes it normal inside $G$, which allows us to take the quotient. As this procedure is usually delicate because the category of finite flat affine group schemes over an arbitrary $R$ is not an abelian category, we state precisely a result of Grothendieck which grants us the existence of the quotient of $G$ by $N$ as a finite flat group scheme over $R$ (note the fundamental hypothesis that $N$ is normal) (see for example section 3 in \cite{Ta97} , or sec. 6.3 in \cite{Sti09} or \cite{Sch00}): 
 
\begin{theorem}
Let $N$ be a finite flat closed normal subgroup of an affine finite group scheme $G$ over $R$. Then the quotient group fpqc sheaf $G/N$ is representable by an affine finite group scheme $H$, which coincides with the categorical cokernel for the inclusion $N \subset G$ with the natural projection $ G \rightarrow G/N$ being finite and faithfully flat.\\
Moreover, if $G=\text{Spec}(A)$ and $N=\text{Spec}(A/J)$ is commutative then $G/N = \text{Spec}(B)$ where the $R$-algebra $B$ is described explicitly by $B=\{a \in A : c(a) \equiv 1 \otimes a \text{ mod } (J \otimes A)\}$, where $c$ denotes the co-multiplication on $A$. 
\end{theorem}
\noindent
The fact that the natural projection $G \rightarrow G/N$ is finite and faithfully flat implies in particular that $G$ is finite flat if and only if $G/N$ is finite and flat (note that we are assuming as hypothesis that $N$ is already finite and flat). 
Coming back to our case, since $G$ is a finite flat deformation of $G_\lambda$ we have that $G/N$ is a finite flat $R$-group scheme. Moreover, we have that $N=\text{Spec} (R[x] / (x^p ))$ with the group law mentioned above is commutative because it is a finite flat group scheme of prime order (this is a result of Oort and Tate, see for example \cite{OT70}). Hence, $G/N$ is isomorphic $ \text{Spec(B)}$ where $B$ is the $R$-subalgebra of $R[x, y] /(x^p -c\pi y^{p^{v+1}} , y^{p^m})$ such that 
$$B=\{a \in R[x, y] /(x^p -c\pi y^{p^{v+1}} , y^{p^m}) : c(a) \equiv 1 \otimes a \text{ mod }((y)\otimes A)\}$$
where as usual, the group law of $G$ is determined by $c(x)= (1+y)^\lambda \otimes x + x \otimes 1$ and $c(y)= 1\otimes y + y \otimes 1 + y\otimes y$ (after Lemma \ref{lemmah}). It is straightforward to check that $B$ is the $R$-subalgebra of $R[x, y] /(x^p -c\pi y^{p^{v+1}} , y^{p^m})$ generated by the elements $x^p$ and $y$ with relation $x^p -c\pi y^{p^{v+1}}=0$. Indeed we have that because the formulas for $c$ holds as above, the element $x$ does not belong to $B$ but the elements $x^p$ and $y$ do (and no smaller power of $x$ belongs to $B$). Note also that $(x^p)^2 =0$ because $\pi^2=0$. 
If $c\not=0$, the relation $x^p -c\pi y^{p^{v+1}}=0$ prevents the $R$-algebra $B$ to be free (note that $R$ is local) which would contradict the fact that $\text{Spec}(B)$ is a finite and flat $R$-group scheme. This allows us to conclude directly that $c$ must be zero.
Another equivalent way to prove that $c=0$ comes from the deformation theory of $\mu_{p^m}$. Indeed, the finite flat $R$-group scheme $G/N$ is a deformation of $\mu_{p^{m}}$, i.e.  $G/N \otimes_R k \cong \mu_{p^m} /k$ (because base change of algebraic groups preserves exactness under flatness hypothesis). However, the $k$-group scheme $\mu_{p^m}$ does not admit non-trivial deformations over a local Artin ring $R$, i.e. we have the following: 
\begin{prop}
Let $R$ be a local Artin ring. Let $H$ be a finite flat $R$-group scheme such that $H \otimes_R k \cong \mu_{p^m}$ then we have that  $H \cong \mu_{p^m} / R$. In other words, all deformations of $\mu_{p^m}$ are trivial for all positive integers $m$. 
\end{prop} 
\begin{proof}
This is a direct consequence of the fact that the category of $R$-group schemes which are of multiplicative type and finite over $R$ is equivalent to the same category over $R/\mathfrak{m}_R R \cong k$, and we know that $H\otimes_R k \cong \mu_{p^m}$ is of multiplicative type. This implies that $H$ and $\mu_{p^m}$ over $R$ have to be isomorphic. See for example Cor. 2.3 and 2.4, Rem. 4.0.1 and Lemma 4.1 in SGA 3 Exp. X. in \cite{SGA3x}.
\end{proof}
\noindent
We conclude that we have the isomorphism $G/N\cong \mu_{p^{m}}$ as $R$-group schemes, which implies that $B\cong R[y]/(y^{p^m })$ which holds if and only if $c=0$.
Finally, for all $\lambda$ we have that:
$$G\cong \text{Spec} (R[x, y] / (x^p , y^{p^m}))$$
with group law given for a certain $a \in k$ by:
$$(x, y) \circ_G (x' y' )= ((1+y)^\lambda x' + x +a \pi W_p (x , x') , y +y' +yy' ).$$
By Lemma \ref{lemmah}, we conclude that if $v_p (\lambda)\not=m-1$ there are no non-trivial deformations of $G_\lambda$ over $R$, i.e. $G \cong G_\lambda \otimes_k R \cong \alpha_p \rtimes_\lambda \mu_{p^{m}}$ which explicitly is the $R$-scheme $R[x, y ] / (x^p , y^{p^m})$ with group structure given by $(x, y ) \circ_{G_{\lambda}} (x', y') = ( (1+y)^{\lambda} x' +x , y + y' + yy' )$. \\
If $v_p (\lambda)= {m-1}$, we have that the deformations of $G_\lambda$ form a $1$-dimensional family $\tilde{H}_{a}$ (with parameter $a\in k$) of non-commutative $R$-group schemes of order $p^{m+1}$. We have that the family of finite flat $R$-group schemes $\tilde{H}_{a}$ can be explicitly described as:
$$\tilde{H}_{a} \cong\text{Spec}(R[x, y ] / (x^p , y^{p^m}))$$
with group law given by:
$$(x, y)\circ_{G} (x' , y' )= ((1+y)^{p^m} x' + x +a\pi W_p (x, x') , y + y' + yy' ). $$
This concludes the proof of Theorem \ref{maindef}. Finally, we can apply Theorem \ref{maindef} to the problem of understanding if any finite flat deformation $G$ over $R$ of $G_\lambda$ is killed by its order. In particular, we conclude with the following: 
\begin{theorem}
Let $R$ be a local Artin ring of positive residue characteristic $p$. Let $G$ be a deformation over $R$ of the $k$-group scheme $G_\lambda$ for any $\lambda\in \{1, \cdots , p^{m-1}\}$. Then $G$ is killed by its order. 
\end{theorem}

\begin{proof}
By Theorem \ref{maindef}, if $v_p (\lambda )\not= m-1$, we have that $G\cong (\alpha_{p} \rtimes_\lambda \mu_{p^m})\times_{\text{Spec}(k)}  \text{Spec}(R)$ over $R$ and we conclude that also $G$ is killed by its order. \\
Now, assume that $v_p (\lambda )=m-1$. By Theorem \ref{maindef}, for any deformation $G$ over $R$ of $G_{p^{m-1}}$ there exist $a  \in k$ such that we have that $G \cong \tilde{H}_{a}$. We can perform now computations as the group law is also explicit in this case. Indeed, given a generic element $h=(x, y) \in \tilde{H}_{a}$, we have that 
$$[p^m] (x, y) = \Big( x \sum^{p^{m} -1}_{k=0} (1+y)^{k p^{m-1}}, y^{p^m} \Big)$$
because $W_p (x, x)=0$ since $x^p=0$. Since $y^{p^m}=0$ the second component is zero. Finally, the same exact computation performed in Section 2 can be repeated here which grants us that also the first component is indeed zero because
$\sum^{p^{m} -1}_{k=0} (1+y)^{k p^{m-1}} = p^{m-1} \sum^{p -1}_{k=0} (1+y)^{k p^{m-1}} =0$ 
because $R$ is of characteristic $p$. This proves that $[p^m]$ kills $\tilde{H}_{a}$ and in particular we deduce that $G$ is killed by its order. This concludes the proof of the theorem.   
\end{proof}

\printbibliography[heading=bibintoc,title={References}]

\end{document}